\documentclass[english]{amsart}

\usepackage{babel}
\usepackage{amstext}
\usepackage{amsmath}
\usepackage{amsfonts}
\usepackage{latexsym}
\usepackage{ifthen}
\newcommand{\lra}{\longrightarrow}

\newcommand\sE{{\mathcal E}}

\newcommand\sH{{\mathcal H}}

\newcommand\sS{{\mathcal S}}
\newcommand\sM{{\mathcal M}}

\newcommand\bR{{\mathbb R}}
\newcommand\bZ{{\mathbb Z}}
\newcommand\bC{{\mathbb C}}
\newcommand\bQ{{\mathbb Q}}

\newcommand\bP{{\mathbb P}}
\newcommand\rh{{\dasharrow}}

\newcounter{lemma}

\newtheorem{lemma1}[lemma]{\setcounter{equation}{0}}

\newenvironment{lemma}{\begin{lemma1}{\bf Lemma.}}{\end{lemma1}}

\newenvironment{theorem}{\begin{lemma1}{\bf Theorem.}}{\end{lemma1}}
\newenvironment{question}{\begin{lemma1}{\bf Question.}}{\end{lemma1}}

\newenvironment{corollary}{\begin{lemma1}{\bf Corollary.}}{\end{lemma1}}

\newenvironment{remark}{\begin{lemma1}{\bf Remark.}\rm}{\end{lemma1}}

\newenvironment{Induction Step}{\begin{lemma1}{\bf Induction Step.}}{\end{lemma1}}
\newenvironment{Proof of Theorem 1.2}{\begin{lemma1}{\bf Proof of Theorem 1.2.}}{\end{lemma1}}

\title[hyperk\"ahler manifolds]
{A remark on dynamical degrees of automorphisms 
of hyperk\"ahler manifolds}

\author[K. Oguiso]{Keiji Oguiso}
\date{\today}

\begin{document}
\dedicatory{To the memory of Professor Masayoshi Nagata}
\begin{abstract} 
We describe all the dynamical degrees of automorphisms of hyperk\"ahler 
manifolds 
in terms of the first dynamical degree. We also present two explicit examples 
of different geometric flavours. 
\end{abstract}

\maketitle

\section{Introduction}
\setcounter{lemma}{0}

A {\it hyperk\"ahler manifold} is a simply connected compact K\"ahler manifold 
$X$ such that $H^0(X, \Omega_X^2) = \bC \sigma_X$, where $\sigma_X$ 
is an everywhere non-degenerate holomorphic $2$-form on $X$. It is then of 
even dimension. 
\par
\vskip 4pt
The aim of this note is to point out the following simple, explicit behaviour 
of the {\it dynamical degrees} of automorphisms of hyperk\"ahler manifolds (see Section 2 for terminologies). 
\begin{theorem}\label{hk}
Let $X$ be a hyperk\"ahler manifold of dimension $2n$ and $g$ be a holomorphic automorphism of $X$. Let $k$ be an integer such that $0 \le k \le 2n$ and $d_k(g)$ be the $k$-th dynamical degree of $g$. Then 
$$d_{2n-k}(g) = d_k(g) = d_1(g)^k$$ 
for $0 \le k \le n$. 
In particular, if $g$ is of positive entropy, then
$$1= d_0(g) < d_1(g) < \cdots < d_{n-1}(g) < d_{n}(g) 
> d_{n+1}(g) > \cdots > d_{2n}(g) = 1\,\, ,$$
i.e., the function 
$$\bZ \cap [0, n] \ni k \mapsto d_k(g) \in \bR$$ 
is strictly increasing. Moreover, the topological entropy $h(g)$ of $g$ is 
$$h(g) = n\log\, d_1(g)\,.$$   
\end{theorem}
This behaviour fits well with the assumption 
made in \cite{DS09} Section 4.4. So, one can now apply the results there 
for automorphisms of hyperk\"ahler manifolds. 
\par
\vskip 4pt
A {\it Salem number} is a real algebraic integer 
$\alpha$ such that $\alpha > 1$, 
$1/\alpha$ is a Galois conjugate of $\alpha$ and all other Galois conjugates 
are of absolute value $1$. Salem numbers play important roles in complex dynamics on K3 surfaces, rational surfaces (\cite{Mc02}, \cite{Mc07}) and hyperk\"ahler manifolds (\cite{Og07} and references therein).

\begin{corollary}\label{salem} 
The dynamical degrees of automorphisms of hyperk\"ahler manifolds 
are Salem numbers or $1$.  
\end{corollary}

Dynamical degrees are natural, important invariants in complex dynamics of automorphisms, but there are very few examples in dimension $\ge 3$ 
for which explicit computations are made. 
\par
\vskip 4pt
The next aim of this note is to give a few explicit computaions of dynamical degrees of automorphisms of $S^{[n]}$ (Examples 1, 2 in Section 4). Here and hereafter $S^{[n]}$ 
is the Hilbert scheme of $0$-dimensional subschemes of length $n$ on a K3 surface $S$. As it is well known, $S^{[n]}$ is a $2n$-dimensional hyperk\"ahler manifold (\cite{Fu83} for $n =2$, \cite{Be83} for any $n \ge 2$). 
\par
\vskip 4pt
Example 1 is elementary and everything is very concrete, but 
automorphisms under consideration are all induced from automorphisms of $S$. In Example 2, we construct automorphisms of positive entropy of some $S^{[2]}$ {\it that are not induced from any automorphisms of any $S'$ such that $(S')^{[2]} \simeq S^{[2]}$}. A similar construction in a more general context was also made by O'Grady \cite{OGr05}, Section 4 (see especially subsection 4.4). For our explicit computation of the dynamical degrees, the specification made in Lemma 4.5 is convenient. 
\par
\vskip 4pt
Theorem \ref{hk} and Examples 1, 2 will partially  
answer to the question of Professor Nessim Sibony to me (June 2008):
 
\begin{question}\label{qnsibony}
Besides automorphisms of complex tori or automorphisms of product 
type (\cite{DS09} Section 4.4), are there any explicit examples of 
automorphisms $f$ of higher dimensional manifolds $M$ 
such that the dynamical degrees 
$d_k(f)$ ($k \in [0, \dim\, M/2] \cap \bZ$) are strictly increasing, i.e., mutually distinct?
\end{question}
In the view of Corollary \ref{salem}, the following question might be of 
interest:
\begin{question}\label{qnsalem}
Which Salem numbers are realizable as dynamical degrees of automorphisms 
of hyperk\"ahler manifolds? In particular, is the Lehmer number, i.e., 
the smallest known Salem number 
$$\lambda_{\rm Lehmer} = 1.17628...$$
realizable in this way? 
\end{question}
We note that the Lehmer number is realizable as the first dynamical degree 
of a rational surface automorphism \cite{Mc07}. 
\par
\vskip 4pt

\section{Dynamical degrees of automorphisms - a few known results}
\setcounter{lemma}{0}
In this section, we briefly review some basic notions 
and some {\it known} properties of dynamical degrees. Our sources are \cite{DS05} Section 2 and \cite{Gu05} Section 1. \cite{Zh08} Section 2 also provides a very good survey particularly for algebraic geometers. For more details or 
any history, see these papers and the references therein.
\par
\vskip 4pt
Let $M$ be a compact K\"ahler manifold of dimension $m$ and $f$ be a biholomorphic automorphism of $M$. Let $k$ be an integer such that $0 \le k \le m$. 
Then, the {\it $k$-th dynamical degree} $d_k(f)$ of $f$ is defined by
$$d_k(f) = {\lim}_{\ell \lra \infty} (\int_M (f^{\ell})^*\eta^k 
\wedge \eta^{m-k})^{1/\ell}\, .$$
Here $\eta$ is a K\"ahler form on $M$. It is well known that $d_k(f)$ does not depend on the choice of $\eta$. Moreover, 
$$d_k(f) = \rho(f^{*} \vert H^{k,k}(M, \bR)) = 
\rho(f^{*} \vert H^{2k}(M, \bR))\, .$$
Here $\rho(*)$ denotes the spectral radius of the linear map $*$. 
Probably, the first important property of $d_k(f)$ 
is the following symmetry:
$$d_{k}(f) = d_{m-k}(f^{-1})\,\, .$$ 
This follows from the Poincar\'e duality. As a more interesting property, it is also known 
that  
the function 
$$[0, m] \ni k \mapsto \log d_k(f) \in \bR$$
is always concave (see eg. \cite{Gu05} Proposition 1.2, \cite{DS05} Proposition 2.6). From these two properties and $d_0(f) = 1$, we have the following general behaviour of the dynamical degrees: 
\begin{theorem}\label{gen}
There are $p, q \in [0, m] \cap \bZ$ such that $p \le q$ and 
$$1 = d_0(f) < \cdots < d_p(f) = \cdots = d_{q}(f) > \cdots > d_m(f) = 1\, .$$
\end{theorem}   
The {\it topological entropy} $h(f)$ of $f$ is a real number that measures the complexity of orbit space
$$x\, ,\, f(x)\, ,\, f^2(x)\, ,\, f^3(x)\, ,\, \cdots\, , f^m(x)\, ,\, \cdots 
\, .$$
Original definition of $h(f)$ is purely topological (\cite{Gr03}, Section 1), 
and the action on the cohomology ring never appears in the definition. However, for an automorphism $f$ of a compact K\"ahler manifold $X$, it turns out that 
$$h(f) = \log d_p(f) = {\rm max}\, \{\log d_k(f)\, \vert \, k \in [0,m] \cap \bZ\}\, ,$$
by the fundamental theorem due to Gromov and Yomdin. In this note, we only need this fact. Note then that $h(f) \ge 0$, and $h(f) > 0$ if and only if $d_1(f) > 1$ and that the function $d_k(f)$ is constantly $1$ when $h(f) = 0$. 
\par
\vskip 4pt
So, our new Theorem \ref{hk} says that {\it only two extreme cases in Theorem \ref{gen} can happen for automorphisms of hyperk\"ahler manifolds. That is, according to $h(g) =0$ or $h(g) > 0$, either $p = 0$ and $q = \dim\, X$ or $p = q = n = \dim\, X/2$}. 

\section{Proof of Theorem \ref{hk} and Corollary \ref{salem}}
\setcounter{lemma}{0}

In this section, we shall prove Theorem \ref{hk} and Corollary \ref{salem}. 
The proof is extremely simple. 
\par
\vskip 4pt
Let $k \in [0, n] \cap \bZ$. Consider the natural map 
$$m_k : {\rm Sym}^{k}H^2(X, \bR) \lra H^{2k}(X, \bR)\, .$$
$m_k$ is given by the multiplication of the cohomology 
ring $H^{*}(X, \bR)$. In particular, $m_k$ is compatible with the action 
of ${\rm Aut}\, X$. Moreover, by a result of Verbitsky \cite{Ve96} (see also \cite{Bo96} together with \cite{Bo78}, \cite{CY03} 
Part III, 
Proposition 24.1), the map $m_k$ is {\it injective} for $k$ in the 
range above. 
Thus $d_k(g) \ge d_1(g)^k$, i.e., 
$$\log d_k(g) \ge k \log d_1(g)\, .$$ On the other hand, as 
$k \mapsto \log d_k(g)$ is a concave function with $\log d_0(g) = 0$, 
we have that 
$$\log d_k(g) \le k \log d_1(g)\, .$$
Thus, $\log d_k(g) = k \log d_1(g)$, and hence $d_k(g) = d_1(g)^k$ 
for each $k \in [0, n] \cap \bZ$. 
\par
\vskip 4pt
On the other hand, by \cite{Og07} Section 2.2, the characteristic polynomial of $g^* \vert H^2(X, \bZ)$ has at most one Salem polynomial factor, i.e., the minimal polynomial of a Salem number 
over $\bZ$, and the other factors are all cyclotomic polynomials. Thus, the set of eigenvalues of $g$ and $g^{-1}$ are the same. We should note that {\it this is a very special property of an automorphism of hyperk\"ahler manifolds}. In particular, 
$$d_1(g) = d_1(g^{-1})\, .$$ 
Hence,  
$$d_{2n-k}(g) = d_{k}(g^{-1}) = d_1(g^{-1})^k = d_1(g)^k\, .$$
This completes the proof of Theorem \ref{hk}.
\par
\vskip 4pt
Let us show Corollary \ref{salem}. The result is clear if $h(g) = 0$. Assume 
that $h(g) > 0$. Then, again by \cite{Og07} Section 2.2, the characteristic polynomial of $g^* \vert H^2(X, \bZ)$ has 
a Salem polynomial factor. Thus, $d_1(g)$ is a Salem number. Let $\gamma$ be a Galois conjugate of $d_1(g)^k$. Then, 
$\gamma$ is of 
the form $\beta^k$ for some Galois conjugate $\beta$ of $d_1(g)$ and vice 
versa. This implies the result. 
\begin{remark}\label{simple}
Let $X$ be a hyperk\"ahler $4$-fold which is deformation equivalent to the Hilbert scheme of a K3 surface, say, $S^{[2]}$. Then $b_2(X) = b_2(S^{[2]}) = 23$ 
and $b_4(X) = b_4(S^{[2]}) = 276$. (This is well known and is indeed easy to see from the explicit description of $S^{[2]}$.) Thus, the injective map $m_2$ 
is actually an isomorphism for dimension reason and we have that $H^4(X, \bR) = {\rm Sym}^{2}H^2(X, \bR)$. Thus, if $g \in {\rm Aut}\, X$ is of positive entropy, then the dynamical degree $d_k(g)$ is an eigenvalue of $g^* \vert H^{2k}(X, \mathbf R)$ {\it of multiplicity one} for each $k$. {\it It would be interesting to see to what extent this multiplicity one property of dynamical degrees holds for automorphisms of hyperk\"ahler manifolds of positive entropy}.
\end{remark}

\begin{remark}\label{amerik}
Recall that the Hilbert scheme $F(V)$ of lines on a smooth cubic $4$-fold $V$ is deformation equivalent to $S^{[2]}$ \cite{BD85}. In \cite{Am07}, Amerik 
computed 
the dynamical degrees of a natural, very interesting self rational map of 
degree $16$ on $F(V)$. They are 
$$d_0 = 1\, ,\, d_1 = 7\, ,\, d_2 = 31\, ,\, d_3 = 28\, ,\, d_4 = 16\, .$$ 
It is remarkable that the behaviour is quite different from the case of automorphisms in Theorem \ref{hk}. This difference is pointed out to us by Professor Fr\'ed\'eric 
Campana. 
\end{remark}

\section{A few explicit examples}
\setcounter{lemma}{0}

Let $S$ be a K3 surface and $S^{[n]}$ be the Hilbert scheme of $n \ge 2$ points on $S$. In this section, we shall compute the dynamical degrees of some automorphisms of $S^{[n]}$ in two specific situations (Examples 1, 2).
\par
\vskip 4pt
\begin{lemma}\label{autk3}
Let $g$ be an automorphism of $S$ and $g_n$ be the automorphism of $S^{[n]}$ 
induced from $g$. Then $d_1(g_n) = d_1(g)$, and 
$d_{2n-k}(g_n) = d_k(g_n) = d_1(g)^k$ 
for $k \in [0, n] \cap \bZ$. In particular, $h(g_n) = n \log\, d_1(g)$. 
\end{lemma}
\begin{proof} Let $E$ be the exceptional divisor of the Hilbert-Chow morphism 
$$S^{[n]} \lra S^{(n)} = S^n/\Sigma_n\, .$$
Then, by \cite{Be83} (see also \cite{CY03} Part III, Section 23), $H^2(S^{[n]}, \bZ)$ carries a natural integral-valued symmetric bilinear form, called the 
{\it Beauville-Bogomolov-Fujiki's form}. Via the natural isomorphism $H^2(S, \bZ) \simeq H^2(S^n, \bZ)^{\Sigma_n}$, we can embed $H^2(S, \bZ)$ into $H^2(S^{[n]}, \bZ)$. \cite{Be83} shows that 
this embedding gives a natural Hodge isometry 
$$H^2(S^{[n]}, \bZ) \simeq H^2(S, \bZ) \oplus \bZ [e]\, .$$
Here $e := E/2 \in H^2(S^{[n]}, \bZ)$ and $(e^2) = -2n + 2$. Since $g_n$ 
is induced from $g$, under the isometry above, we have that $g_n^*(e) = e$, $g_n^*(H^2(S, \bZ)) = H^2(S, \bZ)$ and $g_n^* \vert H^2(S, \bZ) = g^* \vert H^2(S, \bZ)$. Thus, $d_1(g_n) = d_1(g)$. This together with Theorem \ref{hk} implies 
the result. 
\end{proof} 
\noindent
{\it Example 1.} Let $E$ be an elliptic curve. Put $A = E \times E$. Let $S = {\rm Km}\, (A)$ be the Kummer K3 surface associated with $A$, i.e., the minimal 
resolution of the quotient surface 
$A/\langle -1 \rangle$. Let 
$$M := \left( \begin{array}{cc} a& b\\ 
c &  d \end{array}\right) \in\, {\rm SL}\,(2, \bZ)\, .$$
Then $M$ defines an automorphism $f_M$ of $A$ by
$$A \ni (x, y) \mapsto (ax + by, cx + dy) \in A\, .$$
Since $f_M$ commutes with the inversion $-1$ of $A$, 
it descends to an automorphism $g_{M}$ of $S$, and therefore, 
induces an automorphism $g_{M, n}$ on $S^{[n]}$. Recall that  
$$H^2(S, \bQ) \simeq H^2(A, \bQ) \oplus \oplus_{i=1}^{16} \bQ [e_i]\, ,$$
where $e_i$ ($1 \le i \le 16$) are the exceptional divisors of 
$S \lra A/\langle -1 \rangle$. 
This isomorphism is compatible with the actions of $f_M$ 
and $g_M$. In particular, the subspaces 
$\oplus_{i=1}^{16} \bQ [e_i]$ is stable under the action. Moreover, 
the cup product is negative definite on 
$\oplus_{i=1}^{16} \bQ [e_i]$. 
Thus, $d_1(g_M) = d_1(f_M)$. Hence $d_1(g_{M, n}) = d_1(f_M)$ and 
$h(g_{M, n}) = n \log d_1(f_M)$ by Lemma \ref{autk3}. 
\begin{lemma}\label{lin}
Let $t = a +d$, 
$$\alpha = \frac{t + \sqrt{t^2 - 4}}{2}\,\, ,\,\, \beta = \frac{t - \sqrt{t^2 - 4}}{2}\,\, .$$
$\alpha$ and $\beta$ are the eigenvalues of the matrix $M$ 
counted with multiplicities. Then
$$d_{1}(f_M) = 1\,\, {\rm if}\,\, \vert t \vert \le 2\, ,$$ 
$$d_{1}(f_M) = \alpha^2 > 1\,\, {\rm if}\,\, t > 2\, ,$$
$$d_{1}(f_M) = \beta^2 > 1\,\, {\rm if}\,\, t < -2\, .$$
\end{lemma}
\begin{proof} $d_1(f_M)$ is the spectral radius of $f_M^* \vert H^{1,1}(A, \bR)$. Since $A$ is a $2$-dimensional complex torus, we have the following natural isomorphism:
$$H^{1,1}(A, \bR) \otimes \bC = H^{1,1}(A, \bC) \simeq H^0(A, \Omega_A^1) \otimes \overline{H^0(A, \Omega_A^1)}\, .$$
Thus, the eigenvalues of $f_M^* \vert H^{1,1}(A, \bR)$ are $\alpha\overline{\alpha}$, $\alpha\overline{\beta}$, $\beta\overline{\alpha}$, 
$\beta\overline{\beta}$, counted with multiplicities. This implies the result.
\end{proof}
Hence, according to $\vert t \vert \le 2$, $t > 2$, $t < -2$, 
we have 
$$d_{k}(g_{M,n}) = 1\,\, (\forall k \in [0,2n] \cap \bZ)\,\, ,$$ 
$$d_{2n-k}(g_{M,n}) = d_{k}(g_{M,n}) = 
(\frac{t +\sqrt{t^2 -4}}{2})^{2k}\,\, (\forall k \in [0, n] \cap \bZ)\,\, ,$$
$$d_{2n-k}(g_{M,n}) = d_{k}(g_{M,n}) = 
(\frac{t -\sqrt{t^2 -4}}{2})^{2k}\,\, (\forall k \in [0, n] \cap \bZ)\,\, .$$
For the entropy, $h(g_{M, n}) = 0$ in the first case, and 
$h(g_{M, n}) > 0$ in the last two cases.  
\par
\vskip 4pt
\noindent
{\it Example 2.} In this example, we present automorphisms of positive entropy of some $S^{[2]}$ that are not induced from 
automorphisms of any $S'$ such that $(S')^{[2]} \simeq S^{[2]}$. Note that 
there is a pair of K3 surfaces $X$, $Y$ such that $X \not\simeq Y$ but 
$X^{[2]} \simeq Y^{[2]}$ (\cite{Yo01} Example 7.2). 
\par
\vskip 4pt
Recall the natural identification (made in the proof of 
Lemma \ref{autk3})
$$H^2(S^{[2]}, \bZ) = H^2(S, \bZ) \oplus \bZ [e]\, .$$
Under this identification, we have 
$$NS(S^{[2]}) = NS(S) \oplus \bZ [e]\, .$$ 
The involution $\iota$ in the next Lemma is constructed by Beauville 
\cite{Be82}: 
\begin{lemma}\label{line}
Let $S \subset \bP^3$ be a quartic K3 surface. Then the line $\overline{P_1P_2}$ joining two general points $P_1, P_2$ of $S$ meets $S$ in two other points, 
say, 
$P_3$, $P_4$. The correspondence $\{P_1, P_2\} \mapsto \{P_3, P_4\}$ defines 
the birational involution 
$$\iota : S^{[2]} \rh S^{[2]}\, .$$
For this involution, one has:

(1) $\iota$ is a biholomorphic involution if and only if $S$ contains no line.

(2) Under the assumption (1), $\iota^* H = 3H -4e$ and $\iota^*e = 2H - 3e$. Here $H$ is the element of 
$NS(S^{[2]})$ corresponding to the hyperplane class of $S$.

(3) Under the assumption (1), $H^{2}(S^{[2]}, \bZ)^{\iota^*} = \bZ \langle H -e \rangle$. 
\end{lemma}
\begin{proof} (1) is proved by \cite{Be82} Proposition 11. When $NS(S) 
= \bZ H$, i.e., the generic case, (2) is proved by \cite{HT01} Subsection 5.2. 
For arbitrary $S$ with no line, we shall prove by reducing to the generic case 
via deformation. 
\par
\vskip 4pt
Let us consider a general $1$-dimensional small deformation of $S$ inside 
$\bP^3$:
$$\bP^3 \times \Delta \supset \sS \lra \Delta\, ,\, \sS_0 = S\, .$$
Let $\sH$ be the relative hyperplane class of $\sS$ over $\Delta$. 
Then $NS(\sS_t) = \bZ \langle \sH_t \rangle$ for generic $t \in \Delta$. Put  
$$\Delta_2 = \{s \in \Delta\, \vert\, {\rm rank}\, NS(\sS_s) \ge 2\}\, .$$ 
This set $\Delta_2$ is topologically dense in $\Delta$ but it consists of at most countably many points (\cite{Og03} Theorem 1.1). As $\bR$ is an uncountable set, 
for generic $c \in \bR$, the real analytic line 
$$\gamma_c =\{z \in \Delta\, \vert\, {\rm Im}\, z = c \cdot {\rm Re}\, z\} 
\subset \Delta$$ then satisfies 
$$0 \in \gamma_c\,\, {\rm and}\,\, \gamma_c \cap (\Delta_2 \setminus \{0\}) = 
\emptyset\, .$$ 
By construction, $NS(\sS_t) = \bZ \langle \sH_t \rangle$ for $t \in \gamma_c \setminus \{0\}$. 
By taking the relative Hilbert scheme over $\Delta$ and then restricting it
to $\gamma_c$, we have an analytic family 
$$f : \sM \lra \gamma_c$$
such that $\sM_t = \sS_{t}^{[2]}$ for all $t \in \gamma_c$. Let $\sE$ be 
the exceptional divisor of the relative Hilbert-Chow morphism over $\gamma_c$. 
Then, 
$NS(\sM_t) = \bZ \langle \sH_t, e_t 
\rangle$ for all $t \in \gamma_c \setminus \{0\}$ and 
$NS(S^{[2]}) \supset \bZ \langle H, e 
\rangle$ for $t = 0$. Here $e_t = \sE_t/2$, $H = \sH_0$ and $e = e_0$. 
In particular, $\bZ \langle \sH_t, e_t 
\rangle$ ($t \in \gamma_c$) forms a constant subsystem $\Lambda$ 
of the constant system 
$R^2f_*\bZ_{\sM}$. 
\par
\vskip 4pt
As $\sS_t$ ($t \in \gamma_c$) contains no $\bP^1$, the involution $\iota$ 
extends to a (not only birational but also) {\it biholomorphic} involution 
$I$ of 
$\sM$ over $\gamma_c$ such that $I_0 = \iota$. Under the natural action 
of $I$ on $R^2f_*\bZ_{\sM}$, the constant subsystem 
$\Lambda$ is stable. This is because the assertion (2) is true for $\sM_t$ 
($t \in \gamma_c \setminus \{0\}$) by \cite{HT01} Subsection 5.2. Hence, by specializing the result to $t = 0$, we have the assertion (2) also for our 
$S^{[2]}$. 
\par
\vskip 4pt
Let us show the assertion (3). Let $Z \in S^{[2]}$ be a general fixed point of 
$\iota$. Then $Z = P + Q$, where $P \not= Q$ are the tangent points 
of a double tangent line $\ell$ of a general hyperplane $H \cap S$ of $S$ 
(canonical curve of genus $3$). Let $(x_P, y_P)$ be a local coordinate 
of $P$ in $S$ 
such that $x_P$ is the local coordinate of $H\cap S$ at $P$ and $y_P$ 
is the local coordinate of the one parameter small deformation of the double tangent points at $P$. We take a local coordinate $(x_Q, y_Q)$ of $Q$ in $S$ similarly. Consider the local holomorphic $2$-form 
$$\tilde{\sigma} = dx_P \wedge dy_P + dx_Q \wedge dy_Q$$ 
on $S \times S$ at $(P, Q)$. This $\tilde{\sigma}$ is a non-zero 
multiple of $p_1^*\sigma_S + p_2^*\sigma_S$ at $(P, Q)$, possibly after rescalling of $(x_Q, y_Q)$. Here $p_i : S \times S \lra S$ is the projection to the $i$-th factor.
Thus $\tilde{\sigma}$ descends to the local holomorphic $2$-form, say 
$\sigma$, of $S^{[2]}$ at $P + Q$. 
By definition of $\iota$, we have $\iota^*\sigma = -\sigma$. On the other hand, the global holomorphic $2$-form $p_1^*\sigma_S + p_2^*\sigma_S$ 
descends to the global holomorphic $2$-form 
$\sigma_{S^{[2]}}$ on $S^{[2]}$. Hence 
$$\iota^{*}\sigma_{S^{[2]}} = -\sigma_{S^{[2]}}$$ 
at $P + Q$, whence, on $S^{[2]}$ as well. The same is clearly true for the involution $I_t$ on $\sM_t$ ($t \in \gamma_c)$. Thus 
$$I_t^{*}\sigma_{\sM_t} = -\sigma_{\sM_t}\,\, {\rm and}\,\, I_t^{*} \vert 
T(\sM_t) = -1\, .$$ 
Here $T(\sM_t)$ is the {\it transcendental lattice} of $\sM_t$, 
i.e., 
the minimal primitive submodule of $H^{2}(\sM_t, \bZ)$ such that 
$\sigma_{\sM_t} \in T(\sM_t)_{\bC}$. The second equality follows from the first one via the minimality of $T(\sM_t)$. On the other hand,  
$$NS(\sM_t)^{I_t^*} = \bZ \langle \sH_t -e_t \rangle$$  
for all $t \in \gamma_c \setminus \{0\}$, by $NS(\sM_t) = \bZ \langle \sH_t, e_t \rangle$ and by the explicit form of $I_t^{*} \vert NS(\sM_t)$ in (2). As 
$NS(\sM_t) \oplus T(\sM_t)$ 
is of finite index submodule of $H^2(\sM_t, \bZ)$, we have then that
$$H^2(\sM_t, \bZ)^{I_t^*} = \bZ \langle \sH_t -e_t \rangle$$ 
for all $t \in \gamma_c \setminus \{0\}$. 
As $I$ acts on the constant system $R^2f_*\bZ_{\sM}$, by specializing to 
$t=0$, 
we have 
$$H^2(S^{[2]}, \bZ)^{\iota^*} = \bZ \langle H -e \rangle$$ 
as well. This implies the assertion (3). 
\end{proof}

\begin{lemma}\label{lattice}
Let $N = \bZ \langle h_1, h_2 \rangle$ be a lattice whose bilinear form is
$$((h_i, h_j)) = \left( \begin{array}{cc} 4& 8\\ 
8 &  4 \end{array}\right)\, .$$
Then, there are a projective K3 surface $S$ and an isometry $\varphi : N \simeq NS(S)$ such that both $H_i := \varphi(h_i)$ ($i = 1$, $2$) are very ample on 
$S$. This $S$ contains no $\bP^1$.  
\end{lemma}
\begin{proof} By the explicit form, $N$ is an even lattice of signature 
$(1,1)$. Thus, by \cite{Mo84} Corollary 2.9, there is a projective K3 surface 
$S$ 
such that $N \simeq NS(S)$, say by $\varphi$. Let $A(S) \subset NS(S)_{\bR}$ be the ample cone 
of $S$. Again by the explicit form, the lattice $N$ represents neither $0$ nor $-2$. Thus, $S$ has no $\bP^1$, and therefore, $A(S)$ coincides with the positive cone, i.e., one of the two 
connected components of
$$\{x \in NS(S)\, \vert\, (x^2) > 0\}\, .$$
Thus, by replacing $\varphi$ by $-\varphi$ if necessary, both $H_1$ and $H_2$ 
are ample on $S$. They are also very ample. In fact, as $NS(S)$ represents neither $-2$ nor $0$, it follows that $S$ contains neither $\bP^1$ nor elliptic pencil, and we can apply \cite{SD74} Theorem 5.2. 
\end{proof}
{\it From now on, let $S$ be a K3 surface in Lemma \ref{lattice}}. Corresponding to the complete linear systems $\vert H_1 \vert$, $\vert H_2 \vert$, we can embed $S$ into 
$\bP^3$ in two different ways: 
$$\Phi_{\vert H_1 \vert} : S \simeq S_1 \subset \bP^3\, ,\, 
\Phi_{\vert H_2 \vert} :S \simeq S_2 \subset \bP^3\, .$$  
The surfaces $S_1$ and $S_2$ are quartic surfaces in $\bP^3$ with no line. 
Thus, by applying Lemma \ref{line} to $S_1$ and $S_2$ respectively, 
we have two biholomorphic involutions $\iota_1$ and $\iota_2$ on 
$S^{[2]}$. We also note that $NS(S^{[2]}) = \bZ \langle H_1, H_2, e \rangle$. 

\begin{lemma}\label{matrix}
Under the basis $\langle H_1, e, H_2\rangle$ of $NS(S^{[2]})$, the involutions 
$\iota_1^* \vert NS(S^{[2]})$ and $\iota_2^* \vert NS(S^{[2]})$ are represented by the following matrices:
$$M_1 := \left( \begin{array}{ccc} 3 & 2 &8\\ 
-4 & -3 &  -8\\
0 & 0 & -1 \end{array}\right)\, ,\, 
M_2 := \left( \begin{array}{ccc} -1& 0 &0\\ 
-8 & -3 &  -4\\
8 & 2 & 3 \end{array}\right)\, .$$
\end{lemma}
\begin{proof} By Lemma \ref{line}, we see that 
$$\iota_1^*(H_1) = 3H -4e\, ,\, \iota_1^*(e) = 2H_1 - 3e\, .$$
\par
\vskip 4pt 
Let us determine $\iota_1^*(H_2)$. As $NS(S^{[2]}) = \bZ \langle H_1, e, 
H_2 \rangle$, one can write 
$$\iota_1^* H_2 = aH_1 + be + cH_2\, .$$ 
Since $\iota_1^*$ is an isometry, we have 
$$((aH_1 + be + cH_2)^2) = (H_2^2) = 4\, ,$$
$$(3H_1 -4e, aH_1 + be + cH_2) = (H_1, H_2) = 8\, ,$$
$$(2H_1 - 3e, aH_1 + be + cH_2) = (e, H_2) = 0\, .$$
From these three equalities, we 
obtain that 
$$(a, b, c) = (8, -8, -1)\, {\rm or}\, (4, -8, 1)\, .$$
Let us exclude the second case $(4, -8, 1)$. If this happens, then 
$$\iota_1^*(H_2) + H_2 = 4H_1 -8e + 2H_2\, .$$
On the other hand, the class $\iota_1^*(H_2) + H_2$ is 
invariant under $\iota_1$, a contradiction to Lemma \ref{line} 
(3). Thus, 
$$\iota_1^* H_2 = 8H_1 - 8e - H_2\, .$$ 
Hence the matrix representation 
of $\iota_1^* \vert NS(S^{[2]})$ is $M_1$ as claimed. Changing the roles of $H_1$ and $H_2$, we see that the matrix representation 
of $\iota_2^* \vert NS(S^{[2]})$ is $M_2$ as well.
\end{proof}
\begin{lemma}\label{final}
Let $\ell \in [1, \infty) \cap \bZ$ and $g_{\ell} = (\iota_2 \iota_1)^{\ell} \in {\rm Aut}\, S^{[2]}$. Then, $g_{\ell}$ is not induced from any 
${\rm Aut}\, S'$. Here $S'$ is any K3 surface such that 
$(S')^{[2]} \simeq S^{[2]}$. Moreover, we have:
$$d_1(g_{\ell}) = d_3(g_{\ell}) = (17 + 12\sqrt{2})^{\ell}\, ,\, 
d_2(g_{\ell}) = 
(17 + 12\sqrt{2})^{2\ell}\, .$$
\end{lemma}
\begin{proof} By Lemma \ref{matrix}, the matrix representation of 
$(\iota_2\iota_1)^* \vert NS(S^{[2]})$ is 
$$M_1M_2 = \left( \begin{array}{ccc} 45 & 10 &16\\ 
-36 & -7 &  -12\\
-8 & -2 & -3 \end{array}\right)$$
under the basis $\langle H_1, e, H_2 \rangle$ of $NS(S^{[2]})$. 
By straightforward calculation, we see that 
the characteristic polynomial of $M_1M_2$ is 
$$\Phi(x) = (x-1)(x - (17 + 12\sqrt{2}))(x - (17 - 12\sqrt{2}))\, .$$
This together with Theorem \ref{hk} implies the assertion of dynamical 
degrees. Here we used the fact that 
$$\rho(g_{\ell}^* \vert NS(S^{[2]})) 
= \rho(g_{\ell}^* \vert H^{1,1}(S^{[2]}, \bR))\, .$$ 
This equality follows 
from the fact that the Beauville-Bogolomov-Fujiki's form 
on the orthogonal complement of $NS(S)_{\bR}$ in $H^{1,1}(S^{[2]}, \bR)$ 
is negative definite.  
\par
\vskip 4pt
Let us show that $g_{\ell}$ is not induced from any automorphism of $S'$. 
Suppose that $g_{\ell}$ is induced from some automorphism of $S'$. 
Then, $g_{\ell}^*(e') = e'$. Here $E'$ is the exceptional divisor 
of the Hilbert-Chow morphism $(S')^{[2]} \lra (S')^{(2)}$ 
and $e' = E'/2$. On the other hand, by the explicit form of $M_1M_2$ above, 
the eigen vector in $NS(S^{[2]})$ corresponding to the eigenvalue 
$1$ is $k(H_1 - 6e + H_2)$ ($k \in \bZ \setminus \{0\}$).
Thus $e' = k(H_1 - 6e + H_2)$ for some $k \in \bZ$. 
However, this is impossible, as 
$$((e')^2) = -2\,\, {\rm but}\,\, ((H_1 - 6e + H_2)^2) 
= -48\, .$$ 
\end{proof} 

{\it Acknowledgement.} I would like to express my thanks to 
Professors Nessim Sibony and Tien-Cuong Dinh for inviting me to Paris (June 
2008) and for fruitful discussions. I would like to express my thanks to Professors Fr\'ed\'eric Campana, Kieran G. O'Grady and the refreee for important comments. 


\vskip .2cm
\vskip .2cm
\vskip .4cm \noindent

\vskip .2cm \noindent
Keiji Oguiso \\
Department of Economics, Keio University, 4-1-1
Hiyoshi Kohoku-ku, Yokohama, 223-8521, Japan\\ 
oguiso@hc.cc.keio.ac.jp

\end{document}